\documentclass[11pt,reqno]{amsart}
\usepackage[utf8]{inputenc}
\usepackage[T1,T2A]{fontenc}
\usepackage{amsmath,amssymb,amsthm,mathtools}
\usepackage[margin=1in]{geometry}
\usepackage{enumitem}
\usepackage{array}
\usepackage{graphicx}
\usepackage[colorlinks=true,linkcolor=blue,citecolor=blue,urlcolor=blue]{hyperref}

\newtheorem{theorem}{Theorem}
\newtheorem{lemma}[theorem]{Lemma}
\newtheorem{proposition}[theorem]{Proposition}
\newtheorem{corollary}[theorem]{Corollary}
\theoremstyle{definition}
\newtheorem{remark}[theorem]{Remark}
\newtheorem{example}[theorem]{Example}

\newcommand{\N}{\mathbb{N}}
\newcommand{\Z}{\mathbb{Z}}
\newcommand{\floor}[1]{\left\lfloor #1 \right\rfloor}

\newcommand{\abs}[1]{\left| #1 \right|}
\newcommand{\Tar}{\text{Ч}}

\begin{document}

\title{Asymptotics of the Tchoukaillon array and a conjecture of Beluhov}
\author{Shisheng Li}
\address{University of Science and Technology of China, Hefei, China}
\email{shisheng@mail.ustc.edu.cn}
\date{}

\begin{abstract}
The \emph{Tchoukaillon array} is an infinite array of the positive integers,
arising from a one-row Mancala solitaire, in which each positive integer occurs
exactly once. Its zeroth column is the Flavius Josephus sieve and its zeroth row
is the sequence of Tchoukaillon numbers; the asymptotics of these two edges are
classical results of Andersson and of Broline and Loeb. On the basis of numerical
evidence, N.~Beluhov conjectured (as relayed by Knuth) that the general entry
$\Tar_{i,j}$ satisfies $\Tar_{i,j}\approx(\pi i+2j)^2/(4\pi)$ as $i,j\to\infty$.
We prove this conjecture. In fact we establish the stronger uniform estimate
\[
  \Tar_{i,j}=\frac{(\pi i+2j+2)^2}{4\pi}+O\!\big((i+j+1)^{4/3}\big),
\]
in which both constants $\pi$ and $2$ are produced by the array's own recursion
through a Wallis product, independently of the two edge theorems. Equivalently, the
square root of the entry is asymptotically linear,
$\sqrt{\Tar_{i,j}}=\tfrac{\sqrt\pi}{2}\,i+\tfrac{1}{\sqrt\pi}\,(j+1)+O((i+j+1)^{1/3})$,
the linear blend of the two edge growth-rates. As corollaries
we obtain that the level regions $\{\Tar_{i,j}\le V\}$ are triangles up to a
boundary of width $O(V^{1/6})$, and an $O(\sqrt{M})$ algorithm that locates the
row and column of a given integer~$M$.
\end{abstract}

\maketitle

\section{Introduction}\label{sec:intro}

\subsection{Mancala solitaire}
Mancala is a family of ``sowing'' board games played across the world. We are
concerned with a one-row solitaire version, called \emph{Tchouka}, in the form
used by Knuth \cite[\S7.5.1]{knuth14A}. There is a row of pits indexed
$0,1,2,3,\dots$; pit~$k$ holds $p_k$ stones, and there are $s$ stones in all.
Pit~$0$ is a store. A legal move chooses an index $k>0$ with $p_k=k$, removes all
$k$ stones from pit~$k$, and sows them one apiece into the pits
$k-1,k-2,\dots,1,0$ to its left. The object is to move every stone into pit~$0$.

The basic fact \cite{BrolineLoeb1995,knuth14A} is that for each $s$ there is a
\emph{unique} starting configuration, reachable by a greedy construction, from
which the player can win. Consequently every quantity attached to ``the winning
game with $s$ stones''---how many stones sit in pit~$k$, how many times pit~$k$ is
played---is a well-defined function of $s$, and the stones and pits may be
dispensed with in favour of the arithmetic of these functions. For other
combinatorial studies of Tchoukaillon and related solitaires, see
\cite{Dukes2021,JTT2013}.

\subsection{The Tchoukaillon array}
Let $\Tar_n$ denote the least $s$ whose winning configuration uses pit~$n$. The
resulting sequence
\[
  \Tar_1,\Tar_2,\Tar_3,\dots=1,2,4,6,10,12,18,22,30,34,42,48,\dots
\]
is the sequence of \emph{Tchoukaillon numbers}
(OEIS \href{https://oeis.org/A002491}{A002491}). This one-row solitaire and its
name are due to Deledicq and Popova; see \cite{BrolineLoeb1995}.
Knuth \cite[\S7.5.1, Exercise~13]{knuth14A} organises the whole winning structure
into a two-dimensional array. For $n\ge1$ let $\Tar^{(n)}$ be the array of
\emph{order} $n$, with rows $i\ge0$ and columns $0\le j<n$, defined by
\begin{equation}\label{eq:recursion}
\begin{aligned}
  \Tar^{(1)}_{q,0}&=q+1;\\
  \Tar^{(n+1)}_{qn+r,\,j}&=
    \begin{cases}
      \Tar^{(n)}_{q(n+1)+r,\,j}, & j+r<n,\\[2pt]
      \Tar^{(n)}_{q(n+1)+r+1,\,j-1}, & j+r\ge n,
    \end{cases}
    \qquad (i=qn+r,\ 0\le r<n).
\end{aligned}
\end{equation}
Each entry stabilises: $\Tar^{(n)}_{i,j}$ is independent of $n$ once $n\ge i+j+1$.
Indeed, for such $n$ we have $i<n$ and $i+j<n$, so \eqref{eq:recursion} with $q=0$,
$r=i$ takes its first branch and gives $\Tar^{(n+1)}_{i,j}=\Tar^{(n)}_{i,j}$. We write
\begin{equation}\label{eq:limit}
  \Tar_{i,j}:=\Tar^{(\infty)}_{i,j}=\Tar^{(i+j+1)}_{i,j}\qquad(i,j\ge0)
\end{equation}
for the limiting array. Following Knuth, we denote it by the Cyrillic letter
$\Tar$ (Che). Its upper-left corner is
\[
\renewcommand{\arraystretch}{1.15}
\begin{array}{c|cccccccc}
 & j{=}0 & 1 & 2 & 3 & 4 & 5 & 6 & 7\\\hline
i{=}0 & 1 & 2 & 4 & 6 & 10 & 12 & 18 & 22\\
1 & 3 & 5 & 8 & 11 & 16 & 20 & 24 & 32\\
2 & 7 & 9 & 14 & 17 & 23 & 28 & 35 & 40\\
3 & 13 & 15 & 21 & 26 & 33 & 38 & 47 & 53\\
4 & 19 & 25 & 29 & 37 & 44 & 50 & 57 & 68\\
5 & 27 & 31 & 41 & 45 & 55 & 64 & 74 & 81\\
6 & 39 & 43 & 51 & 62 & 69 & 75 & 93 & 98
\end{array}
\]

Figure~\ref{fig:orders} draws the first five orders $\Tar^{(1)},\dots,\Tar^{(5)}$.
Passing from $\Tar^{(n)}$ to $\Tar^{(n+1)}$, recurrence \eqref{eq:recursion}
identifies each order-$(n+1)$ cell with one order-$n$ cell of equal value, so a
fixed value simply \emph{migrates} between successive arrays; for instance the
value $8$ travels $(7,0)\to(3,1)\to(2,1)\to(1,2)$ as $n$ runs through
$1,2,3,4$. A cell attains its final value as soon as $n\ge i+j+1$---that is, once
it reaches the anti-diagonal $i+j=n-1$ (shaded)---and never changes afterwards;
the value $8$, for example, is fixed at $(1,2)$ from order~$4$ on. The shaded
triangle of stabilised entries grows with $n$ and in the limit fills the whole
array $\Tar$.

\begin{figure}[t]
\centering
\includegraphics[width=\textwidth]{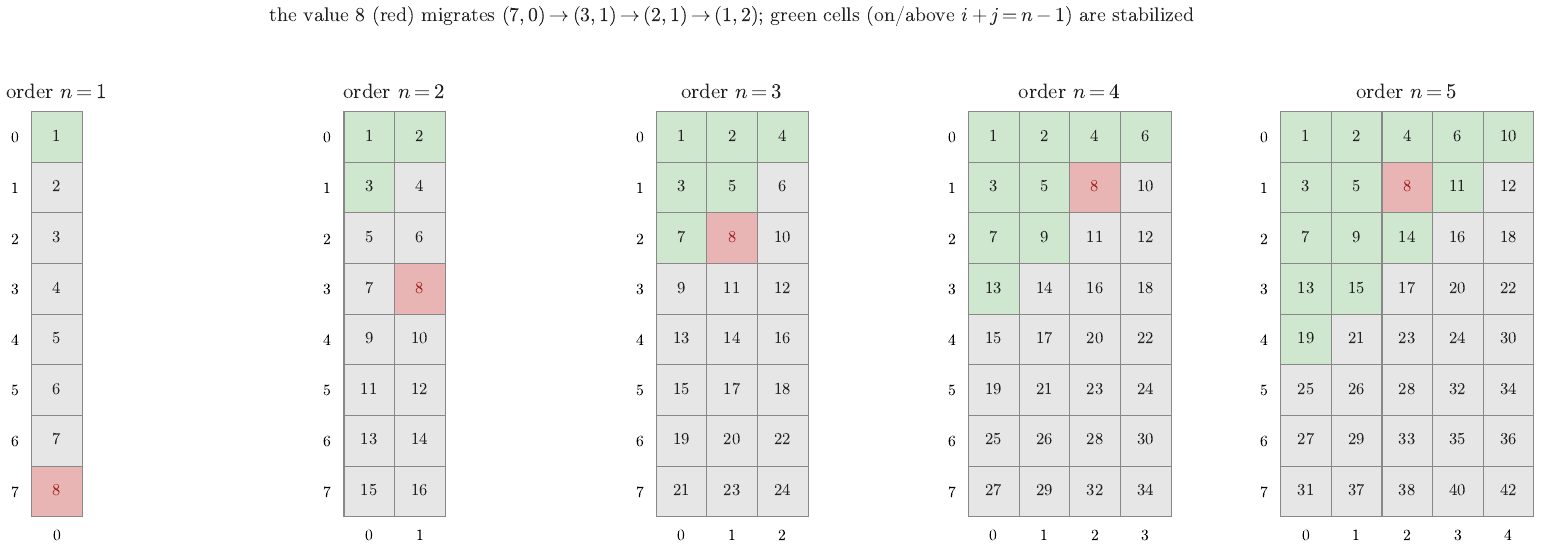}
\caption{The first five orders $\Tar^{(1)},\dots,\Tar^{(5)}$. Green cells are
already stabilised ($i+j+1\le n$, i.e.\ on or above the anti-diagonal
$i+j=n-1$): their values are final and equal to $\Tar_{i,j}$. Grey cells still
change at higher orders. The recurrence \eqref{eq:recursion} moves each value
between consecutive arrays; the value $8$ (red) migrates
$(7,0)\to(3,1)\to(2,1)\to(1,2)$ and stays fixed once it enters the green region
at order~$4$.}
\label{fig:orders}
\end{figure}

\subsection{Known properties}\label{sub:known}
Three facts about $\Tar$ are known; here and below $\N=\{0,1,2,\dots\}$.

\begin{enumerate}[label=\textup{(P\arabic*)},leftmargin=3em]
\item\label{P:bijection}
\emph{Bijection.} $\Tar:\N^2\to\Z^{+}$ is a bijection: each positive integer
occurs exactly once \cite[\S7.5.1]{knuth14A}. Hence
$\#\{(i,j):\Tar_{i,j}\le V\}=\floor{V}$ for every real $V\ge0$ (and $=V$ for integer $V$).
\item\label{P:monotone}
\emph{Monotonicity.} Every row and every column is strictly increasing
\cite[\S7.5.1]{knuth14A}. Thus $\{\Tar_{i,j}\le V\}$ is a lattice down-set.
\item\label{P:edges}
\emph{The two edges.} Column~$0$ is the Flavius Josephus sieve
(OEIS \href{https://oeis.org/A000960}{A000960}) and, by Andersson
\cite{Andersson1998},
\begin{equation}\label{eq:andersson}
  \Tar_{i,0}=\tfrac{\pi}{4}\,i^2+O(i^{4/3}).
\end{equation}
Row~$0$ is the sequence of Tchoukaillon numbers, and by Broline and Loeb
\cite[Thm.~9]{BrolineLoeb1995},
\begin{equation}\label{eq:brolineloeb}
  \Tar_{0,j}=\Tar_{j+1}=\tfrac{1}{\pi}\,(j+1)^2+O(j+1)
\end{equation}
(the single-subscript $\Tar_{j+1}$ being the row-$0$ sequence above).
\end{enumerate}

\subsection{Beluhov's conjecture and our results}
Knuth records Beluhov's conjecture verbatim in the answer to
\cite[\S7.5.1, Exercise~13]{knuth14A}:
\begin{quote}
``Numerical evidence for small $i$ and $j$ has led N.\ Beluhov to conjecture that we
will have $\Tar_{i,j}^{(\infty)}\approx(\pi i+2j)^2/(4\pi)$ in general.''
\end{quote}
The same answer poses two further questions that our method also settles---``Is there
a rapid way to determine which row and column contains a given integer?'' and ``What
is the asymptotic shape of the region occupied by numbers $\le n$?''

\noindent
The conjecture is consistent with \eqref{eq:andersson} on setting $j=0$ and with
\eqref{eq:brolineloeb} on setting $i=0$. Our main result proves it, with an
explicit error term and a naturally centred form of the main term. Throughout we write
\begin{equation}\label{eq:Ndef}
  N:=i+j+1.
\end{equation}

\begin{theorem}\label{thm:main}
As $N=i+j+1\to\infty$,
\begin{equation}\label{eq:main}
  \Tar_{i,j}=\frac{(\pi i+2j+2)^2}{4\pi}+O\!\big(N^{4/3}\big),
\end{equation}
uniformly over all $i,j\ge0$, with an absolute implied constant. In particular
$\Tar_{i,j}=(\pi i+2j)^2/(4\pi)+o(N^2)$, which is Beluhov's conjecture.
\end{theorem}

Equivalently, and more transparently, the \emph{square root} of the entry is
asymptotically linear,
\begin{equation}\label{eq:sqrtintro}
  \sqrt{\Tar_{i,j}}=\frac{\sqrt\pi}{2}\,i+\frac{1}{\sqrt\pi}\,(j+1)+O\!\big(N^{1/3}\big),
\end{equation}
the two slopes being the growth rates of the two edges---Andersson's column~$0$ and
Broline--Loeb's row~$0$---and the interior their linear blend. This \emph{square-root
law} is the geometric heart of the theorem: it is exactly why the level contours are
straight (Corollary~\ref{cor:contour}). The proof does not use
\ref{P:bijection}--\ref{P:edges}; it derives the constants $\pi$ and $2$ from the
recursion \eqref{eq:recursion} alone, through the central binomial coefficients
(a Wallis product) forced by the width recurrence of \S\ref{sec:filter}. Two
corollaries follow.

\begin{corollary}[Contour linearity]\label{cor:contour}
Write $\rho=\pi i+2j$ and $\mathcal R_V=\{(i,j)\in\N^2:\Tar_{i,j}\le V\}$. There
is an absolute constant $K$ such that, for all large integers $V$,
\[
  \big\{\rho\le 2\sqrt{\pi V}-2-KV^{1/6}\big\}\subseteq\mathcal R_V\subseteq
  \big\{\rho\le 2\sqrt{\pi V}-2+KV^{1/6}\big\}.
\]
Thus the level regions are triangles up to a boundary layer of width
$O(V^{1/6})$, answering (in a strong quantitative form) Knuth's question on the
asymptotic shape of $\{\Tar_{i,j}\le n\}$.
\end{corollary}

\begin{proposition}[Locating an integer]\label{prop:locate}
Given a positive integer $M$, set $(X,Y)\leftarrow(M-1,M)$ and $k\leftarrow1$.
While $Y\ne k$, increment $k\leftarrow k+1$ and update, simultaneously,
\(
  (X,Y)\leftarrow\big(X-\floor{Y/k},\,Y-\floor{X/k}\big).
\)
On exit $M=\Tar_{i,j}$ with $i=X$ and $j=k-X-1$. The loop halts after
$N-1=i+j=\Theta(\sqrt{M})$ iterations (none when $M=1$), so determining the row and
column of $M$ costs $O(\sqrt{M})$ arithmetic operations (in the unit-cost model; the
bit complexity carries the usual cost of the integer divisions).
\end{proposition}

Proposition~\ref{prop:locate} answers Knuth's question whether there is a rapid
way to determine which row and column contain a given integer; it is an immediate
reading of the trace recurrence of \S\ref{sec:trace}, run forwards from a given
value.

\begin{figure}[t]
\centering
\includegraphics[width=0.60\textwidth]{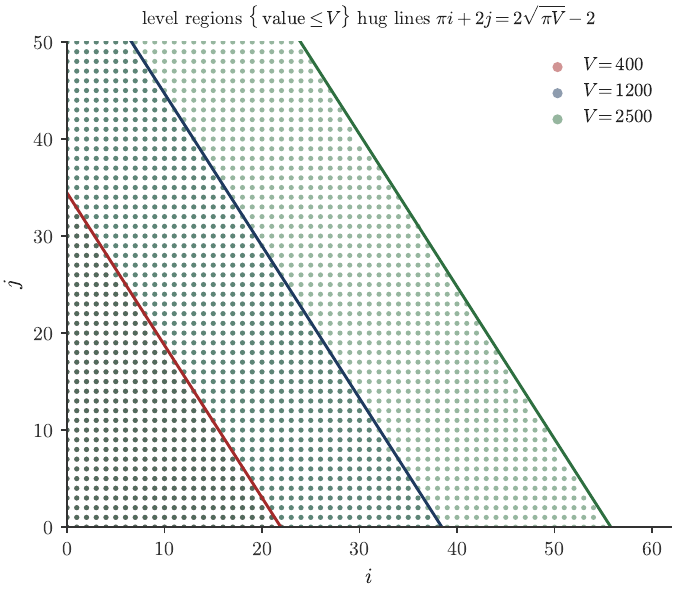}
\caption{The level regions $\{\Tar_{i,j}\le V\}$ (dots) for $V=400,1200,2500$,
each hugging the straight line $\pi i+2j=2\sqrt{\pi V}-2$
(Corollary~\ref{cor:contour}). Since each region is a triangle up to a thin
boundary layer, and by \ref{P:bijection} it contains exactly $V$ cells, its area
$c^2/(4\pi)=V$ forces $c=2\sqrt{\pi V}$; this is the geometry behind
Theorem~\ref{thm:main}.}
\label{fig:contour}
\end{figure}

\subsection{Method and intuition}
The proof rests on one idea: \emph{follow a single integer.} Fix the value
$M=\Tar_{i,j}$ and watch which cell it occupies as the order $k$ of the array
grows from $1$---where $M$ sits at $(M-1,0)$---up to $k=N$, where it has arrived
at its final cell $(i,j)$ (Figure~\ref{fig:trace}). Written in the coordinates
$X=\text{row}$ and $Y=\text{row}+\text{column}+1$, this migration obeys an
\emph{exact} coupled floor recurrence (Lemma~\ref{lem:trace})---Beluhov's problem
restated with the array, and the game, stripped away.

At each step the two coordinates drop by amounts $a_k,b_k$; how long each drop
stays above a level $m$ is recorded by the \emph{conjugate widths} $R_m,S_m$, the
transpose of the drop staircase (Figure~\ref{fig:conjugate}). Read off at the orders
$k=R_m$ where the staircase crosses a level, the floors do not disappear---they become
\emph{nearest-integer rounding} in a lower-triangular system for the widths
(Lemma~\ref{lem:filter}). Removing only this
rounding diagonalises the system into two scalar \emph{Wallis modes}: a contracting
mode driven by the gap datum $j+1$, and its reciprocal driven by the row datum $i$
(Lemma~\ref{lem:modes}). Their sizes are central binomial coefficients, and it is
there, through a Wallis product, that $\pi$ enters---from the recurrence itself, not
from the two edge theorems. Restoring the rounding moves each width by less than one
(Lemma~\ref{lem:track}), so the integer staircase stays within one cell of the
explicit profile at every level. Reading $M$ off a single level, the whole
argument condenses to one line:
\begin{equation}\label{eq:route}
  \sqrt{M}\ \approx\ \sqrt m\,R_m\ \approx\ (j+1)\sqrt m\,c_{m-1}+\frac{i}{2\sqrt m\,c_m}
  \ \longrightarrow\ \frac{j+1}{\sqrt\pi}+\frac{i\sqrt\pi}{2},
\end{equation}
where $c_m=\binom{2m}{m}/4^m$; balancing two errors at $m\asymp N^{2/3}$ turns
``$\approx$'' into the $O(N^{1/3})$ of \eqref{eq:sqrtintro}. Geometrically the
argument computes Figure~\ref{fig:contour}: by the bijection~\ref{P:bijection} a
value equals the number of cells beneath its contour, essentially the area under a
straight line $\pi i+2j=\text{const}$. Table~\ref{tab:notation} collects the
recurring symbols.

\begin{table}[t]
\centering
\renewcommand{\arraystretch}{1.3}
\begin{tabular}{@{}c >{\raggedright\arraybackslash}p{0.84\textwidth}@{}}
\hline
symbol & meaning \\ \hline
$X_k,\ Y_k$ & coordinates of the traced value at order $k$: $X_k$ its row,
   $Y_k=\text{row}+\text{column}+1$ \\
$a_k,\ b_k$ & the amounts by which $X,Y$ drop at step $k$
   ($a_k=\floor{Y_{k-1}/k}$, $b_k=\floor{X_{k-1}/k}$) \\
$R_m,\ S_m$ & conjugate widths, defined as counts \eqref{eq:width}: when the level is reached, the last order at which $a_k\ge m$ (resp.\ $b_k\ge m$), else $1$ \\
$E_m,\ F_m$ & nearest-integer remainders in the width recurrence ($1-m\le E_m\le m-1$, $1-m\le F_m\le m$) \\
$\bar R_m,\ \bar S_m$ & the error-free profile: $\bar R_m=(j+1)c_{m-1}+iw_m$, $\bar S_m=(j+1)c_m+iw_m$ \\
$c_m,\ w_m$ & $c_m=\binom{2m}{m}/4^m$ and $w_m=1/(2mc_m)$: the two Wallis modes that supply $\pi$ \\
$\mathcal C(\alpha)$ & the limiting density: $\Tar_{i,j}=N^2\mathcal C(i/N)+O(N^{4/3})$, with $\alpha=i/N$ and $\mathcal C$ from \eqref{eq:Cdef} \\
\hline
\end{tabular}
\caption{Notation, with intuitive meanings.}
\label{tab:notation}
\end{table}

The remainder of the paper is organised as follows.
\S\ref{sec:trace} establishes the trace recurrence and
Proposition~\ref{prop:locate}. \S\ref{sec:mass} proves the two-sided estimate.
\S\ref{sec:filter} proves the no-skipping lemma and the width recurrence.
\S\ref{sec:wallis} solves it in closed form as two Wallis modes.
\S\ref{sec:sqrt} proves the square-root law, hence Theorem~\ref{thm:main}, and
\S\ref{sec:contour} Corollary~\ref{cor:contour}. \S\ref{sec:remarks} collects open
problems.

\section{The exact trace recurrence}\label{sec:trace}

Fix $i,j\ge0$, put $N=i+j+1$ and $M=\Tar_{i,j}$. Because
$\Tar^{(1)}_{q,0}=q+1$, the value $M$ occupies position $(M-1,0)$ in the array of
order~$1$. We trace this same value through the arrays of orders
$1,2,\dots,N$; let its position in $\Tar^{(k)}$ be $(I_k,J_k)$, so that
\[
  (I_1,J_1)=(M-1,0),\qquad (I_N,J_N)=(i,j).
\]
Introduce
\begin{equation}\label{eq:XY}
  X_k:=I_k,\qquad Y_k:=I_k+J_k+1,
\end{equation}
so that
\begin{equation}\label{eq:XYbc}
  (X_1,Y_1)=(M-1,M),\qquad (X_N,Y_N)=(i,N),
\end{equation}
and, since $0\le J_k<k$,
\begin{equation}\label{eq:gap}
  1\le Y_k-X_k\le k.
\end{equation}

\begin{figure}[t]
\centering
\includegraphics[width=0.58\textwidth]{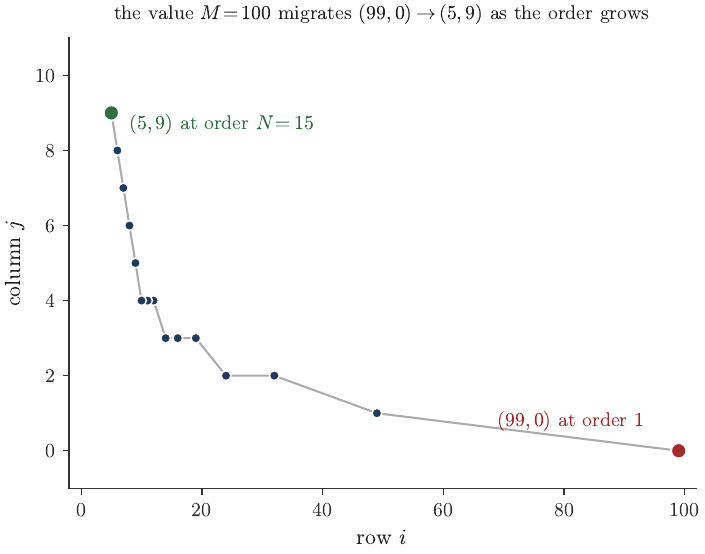}
\caption{Following the value $M=100$: the cell $(I_k,J_k)$ it occupies migrates
from $(99,0)$ at order $1$ to its final position $(5,9)$ at order $N=15$. In the
coordinates $X_k=I_k$, $Y_k=I_k+J_k+1$ this path obeys the exact recurrence
\eqref{eq:trace}.}
\label{fig:trace}
\end{figure}

\begin{lemma}[Trace recurrence]\label{lem:trace}
For $2\le k\le N$,
\begin{equation}\label{eq:trace}
  X_k=X_{k-1}-\floor{\frac{Y_{k-1}}{k}},\qquad
  Y_k=Y_{k-1}-\floor{\frac{X_{k-1}}{k}}.
\end{equation}
\end{lemma}

\begin{proof}
Suppose the traced value occupies the cell $(x,j)$ of the order-$n$ array, so
that $X:=X_{k-1}=x$ and $Y:=Y_{k-1}=x+j+1$, with $0\le j<n$. Write $k=n+1$ and let
\begin{equation}\label{eq:div}
  x=q(n+1)+s,\qquad q=\floor{\tfrac{x}{n+1}},\quad 0\le s\le n,
\end{equation}
be the division of $x$ by $k=n+1$.

\emph{Step 1 (the image cell in the order-$(n+1)$ array).}
Recurrence \eqref{eq:recursion} equates each order-$(n+1)$ entry with one order-$n$
entry, and so matches the cells of the two arrays one-to-one; the traced value
passes to the order-$(n+1)$ cell whose right-hand side in \eqref{eq:recursion} is
$(x,j)$. There are two cases.
\begin{itemize}
\item If $(x,j)$ is a first-line right-hand side $(q(n+1)+r,\,j)$ with
      $0\le r<n$, then \eqref{eq:div} forces $r=s$ (hence $s<n$); the column is
      unchanged, and the condition $j+r<n$ becomes $s+j<n$. The image is the
      left-hand side $(qn+s,\,j)$.
\item If $(x,j)$ is a second-line right-hand side $(q(n+1)+r+1,\,d-1)$, where $d$
      denotes the order-$(n+1)$ column, then $j=d-1$ (so $d=j+1$) and $r+1=s$
      (hence $r=s-1$ and $s\ge1$); the condition $d+r\ge n$ becomes
      $(j+1)+(s-1)=s+j\ge n$. The image is the left-hand side
      $(qn+r,\,d)=(qn+s-1,\,j+1)$.
\end{itemize}
The alternatives $s+j<n$ and $s+j\ge n$ are exhaustive and mutually exclusive, so
exactly one case applies, and
\begin{equation}\label{eq:forward}
  (x,j)\longmapsto
  \begin{cases}
    (qn+s,\ j), & s+j<n,\\
    (qn+s-1,\ j+1), & s+j\ge n .
  \end{cases}
\end{equation}
When $s+j\ge n$, $j\le n-1$ forces $s\ge n-j\ge1$, so the new row index
$qn+s-1\ge0$.

\emph{Step 2 (the two floors).}
By \eqref{eq:div}, $X=x=qk+s$ with $0\le s\le n=k-1<k$, so $\floor{X/k}=q$. Next
$Y=x+j+1=qk+(s+j+1)$, and since $0\le s\le k-1$ and $0\le j\le n-1=k-2$ we have
$1\le s+j+1\le 2k-2<2k$. Hence $\floor{Y/k}=q+\floor{(s+j+1)/k}$, where the last
floor is $0$ when $s+j<n$ (then $s+j+1\le k-1$) and $1$ when $s+j\ge n$ (then
$k\le s+j+1\le 2k-2$). In summary,
\begin{equation}\label{eq:floors}
  \floor{X/k}=q\ \text{always},\qquad
  \floor{Y/k}=\begin{cases} q, & s+j<n,\\[2pt] q+1, & s+j\ge n. \end{cases}
\end{equation}

\emph{Step 3 (matching \eqref{eq:trace}).}
Write the image \eqref{eq:forward} as $(x',j')$ and set $X_k=x'$,
$Y_k=x'+j'+1$; then $(X_k,Y_k)$ is the value's coordinate pair at order $n+1$.
Using $n=k-1$, i.e.\ $qn=qk-q$:
\begin{itemize}
\item if $s+j<n$, then $x'=qn+s=(qk+s)-q=X-q$ and
      $x'+j'+1=qn+s+j+1=(qk+s+j+1)-q=Y-q$; by \eqref{eq:floors},
      $\floor{Y/k}=\floor{X/k}=q$, so $X_k=X-\floor{Y/k}$ and $Y_k=Y-\floor{X/k}$;
\item if $s+j\ge n$, then $x'=qn+s-1=(qk+s)-(q+1)=X-(q+1)$ and
      $x'+j'+1=(qn+s-1)+(j+1)+1=qn+s+j+1=Y-q$; by \eqref{eq:floors},
      $\floor{Y/k}=q+1$ and $\floor{X/k}=q$, so again $X_k=X-\floor{Y/k}$ and
      $Y_k=Y-\floor{X/k}$.
\end{itemize}
Both cases give \eqref{eq:trace}.
\end{proof}

Thus Beluhov's problem is equivalent to the analysis of the coupled floor
recurrence \eqref{eq:trace} with boundary data \eqref{eq:XYbc}. We extend the
trajectory to all $k\ge N$ by its stationary value $(X_k,Y_k)=(i,N)$ (the array has
stabilised); then $a_k=\floor{Y_{k-1}/k}$ and $b_k=\floor{X_{k-1}/k}$ vanish for
$k>N$, since $i,N<k$, so every sum over $k$ below is finite.

\begin{remark}[The symmetry]\label{rem:symmetry}
The coordinates \eqref{eq:XY} were chosen for one reason, now visible in
\eqref{eq:trace}. In the original row and column variables the value's motion
\eqref{eq:forward} splits into two cases and treats $i$ and $j$ asymmetrically;
in $(X,Y)$ it collapses to a single \emph{symmetric} rule, each of $X,Y$
decreasing by the floor of the \emph{other} over $k$:
\[
  X_k=X_{k-1}-\floor{\tfrac{Y_{k-1}}{k}},\qquad
  Y_k=Y_{k-1}-\floor{\tfrac{X_{k-1}}{k}}.
\]
This $X\leftrightarrow Y$ symmetry is the structural heart of the paper. It is
what produces the twin widths $R_m,S_m$ and their paired recurrences
(\S\ref{sec:filter}), while the symmetric sum $X_k+Y_k$ telescopes to the
two-sided estimate (\S\ref{sec:mass}); and it is a coupled subtractive,
Euclidean-type
recurrence---a shape in which the problem becomes analysable, as none is visible
in the $(i,j)$ form. The one asymmetry that remains, the initial gap
$Y_1-X_1=1$, is exactly what the exponent $j$ (a column index) will track through
$Y_k-X_k=J_k+1$.
\end{remark}

\begin{example}\label{ex:trace}
For $M=100$, iterating \eqref{eq:trace} from $(X_1,Y_1)=(99,100)$ gives
\[
 (99,100)\to(49,51)\to(32,35)\to(24,27)\to(19,23)\to\cdots\to(6,15)\to(5,15),
\]
halting at $k=15=N$ (where $Y_{15}=15=k$); hence $M=\Tar_{5,9}$, and indeed
$\Tar_{5,9}=100$. This is the path of Figure~\ref{fig:trace}. Read as an
algorithm---iterate until $Y=k$, then output $(i,j)=(X,\,k-X-1)$---it is
Proposition~\ref{prop:locate}.
\end{example}

\begin{proof}[Proof of Proposition~\ref{prop:locate}]
Run \eqref{eq:trace} forwards from $(X_1,Y_1)=(M-1,M)$; by \eqref{eq:XYbc} the
trajectory reaches $(X_N,Y_N)=(i,N)$ at $k=N$. At each step the quantity $Y_k-k$
decreases by $\floor{X_{k-1}/k}+1\ge1$, so it is strictly decreasing; it equals
$Y_1-1=M-1$ at $k=1$ and $Y_N-N=0$ at $k=N$, and being integer-valued and
strictly decreasing it is positive for every $k<N$. Hence $k=N$ is the first
index with $Y_k=k$, and then $i=X_N$ and $j=N-X_N-1$. The while loop performs
$N-1$ updates before halting at $k=N$. By \eqref{eq:massN}, $2M-1\le N(i+N)<2N^2$, so
$M\le N^2$ and $N\ge\sqrt{M}$. For the matching $N=O(\sqrt M)$ we need only a quadratic
\emph{lower} bound on $M$, and \S\ref{sec:filter}--\S\ref{sec:wallis} supply an
elementary one: by Lemma~\ref{lem:track}, $R_3>\bar R_3-1=(j+1)c_2+iw_3-1\ge\tfrac38N-1$
(as $c_2=\tfrac38$, $w_3=\tfrac8{15}$), so $R_3\ge7=2\cdot3+1$ once $N\ge22$; the
single-level reading \eqref{eq:crossread} at $m=3$ then gives
$M\ge R_3^2\ge(\tfrac38N-1)^2$, i.e.\ $N=O(\sqrt M)$ (the finitely many $N<22$ are
trivial). This uses only \S\ref{sec:filter}--\S\ref{sec:wallis}, not
Theorem~\ref{thm:main}. Hence the algorithm uses $\Theta(\sqrt{M})$ arithmetic operations.
\end{proof}

\section{The drops and a two-sided estimate}\label{sec:mass}

Everything downstream is built from the recurrence's two \emph{drops}---the
amounts by which the coordinates fall at each step. For $k\ge2$ set
\begin{equation}\label{eq:ab}
  a_k:=\floor{\frac{Y_{k-1}}{k}},\qquad b_k:=\floor{\frac{X_{k-1}}{k}},
\end{equation}
so that $X_k=X_{k-1}-a_k$ and $Y_k=Y_{k-1}-b_k$. This short section establishes the
two facts we need about them: they are monotone, and a two-sided estimate already
caps $M$ at order $N^2$.

Two facts carry the section. First, the drops \emph{shrink}: $a_k,b_k$ are
nonincreasing and stay within $1$ of each other (Lemma~\ref{lem:ab}), so the
coordinates fall like $\sim M/k$ in ever-shorter steps. Second, the combination
$k(X_k+Y_k)$ is a near-invariant, pinned near $2M-1$ (Lemma~\ref{lem:mass}), which
caps $M$ at $O(N^2)$.

\begin{lemma}[Monotonicity]\label{lem:ab}
The sequences $(a_k)$ and $(b_k)$ are nonincreasing, and
\begin{equation}\label{eq:abdiff}
  a_k-b_k\in\{0,1\}.
\end{equation}
\end{lemma}

\begin{proof}
Since $X_k\le X_{k-1}$, $Y_k\le Y_{k-1}$ and the denominator increases, $(a_k)$
and $(b_k)$ are nonincreasing. By \eqref{eq:gap} at $k-1$,
$0\le Y_{k-1}-X_{k-1}\le k-1$, so $\floor{Y_{k-1}/k}$ and $\floor{X_{k-1}/k}$
differ by $0$ or $1$, which is \eqref{eq:abdiff}.
\end{proof}

So each drop is no larger than the last, and the coordinates fall by less and less;
the estimate below turns ``$\sim M/k$'' into a precise band.

\begin{lemma}[Two-sided estimate]\label{lem:mass}
For every $k$ with $1\le k\le N$,
\begin{equation}\label{eq:mass}
  2M-1\le k(X_k+Y_k)\le 2M-1+k(k-1).
\end{equation}
In particular, at $k=N$,
\begin{equation}\label{eq:massN}
  N(i+1)\le 2M-1\le N(i+N).
\end{equation}
\end{lemma}

\begin{proof}
Consider the quantity $k(X_k+Y_k)$. Writing the Euclidean divisions
$Y_{k-1}=k\,a_k+r^Y_k$ and $X_{k-1}=k\,b_k+r^X_k$ with remainders
$r^X_k,r^Y_k\in[0,k-1]$, and using $X_k=X_{k-1}-a_k$, $Y_k=Y_{k-1}-b_k$,
\[
  k(X_k+Y_k)=k(X_{k-1}+Y_{k-1})-k(a_k+b_k)
  =(k-1)(X_{k-1}+Y_{k-1})+r^X_k+r^Y_k .
\]
So from order $k-1$ to $k$ it increases by exactly $r^X_k+r^Y_k\in[0,2(k-1)]$.
Starting from $1\cdot(X_1+Y_1)=2M-1$ and accumulating these increments,
\[
  k(X_k+Y_k)=2M-1+\sum_{\ell=2}^{k}(r^X_\ell+r^Y_\ell),\qquad
  0\le\sum_{\ell=2}^{k}(r^X_\ell+r^Y_\ell)\le\sum_{\ell=2}^{k}2(\ell-1)=k(k-1),
\]
which is \eqref{eq:mass}. At $k=N$, $X_N=i$ and $Y_N=N$ give \eqref{eq:massN}.
\end{proof}

In particular $N\to\infty$ forces $M\to\infty$, and \eqref{eq:massN} gives the upper
scale $M=O(N^2)$ (its lower bound $2M-1\ge N(i+1)$ is only linear when $i$ is small,
so a matching $M=\Omega(N^2)$ is not claimed here---it will follow from
Theorem~\ref{thm:main}). The rest of the paper computes the constant.

\section{Conjugate widths and their recurrence}\label{sec:filter}

We now \emph{reverse} the direction of \S\ref{sec:mass}. There we were handed a
value $M$ and ran the recurrence forward to find its cell (Proposition~\ref{prop:locate});
Beluhov's question is the opposite---given the cell $(i,j)$, how large is
$M=\Tar_{i,j}$? So we start where the position is known, at the endpoint
$(X_N,Y_N)=(i,N)$, and telescope \eqref{eq:trace} \emph{down} from it:
\begin{equation}\label{eq:tails}
  X_{k-1}=i+\sum_{\ell\ge k}a_\ell,\qquad Y_{k-1}=N+\sum_{\ell\ge k}b_\ell .
\end{equation}
Carried all the way to order~$1$, where $X_1=M-1$ and $Y_1=M$, this telescope
\emph{is} the answer:
\begin{equation}\label{eq:Msum}
  M=N+\sum_{k\ge2}b_k=i+1+\sum_{k\ge2}a_k .
\end{equation}
So $M$ is nothing but the total of the drops---the \emph{area} under the
drop-staircase $a_k$ of \S\ref{sec:mass}. The whole task is to evaluate this area.

\subsection*{The drops obey a recurrence---and the floor is the difficulty}

Substituting the tail sums \eqref{eq:tails} into \eqref{eq:ab} gives the drops a
recurrence of their own,
\begin{equation}\label{eq:abrec}
  a_k=\Big\lfloor\tfrac1k\big(N+\textstyle\sum_{\ell\ge k}b_\ell\big)\Big\rfloor,
  \qquad
  b_k=\Big\lfloor\tfrac1k\big(i+\textstyle\sum_{\ell\ge k}a_\ell\big)\Big\rfloor,
\end{equation}
which, run downward from $(i,N)$, determines every $a_k,b_k$ and hence $M$. The
trouble is the floor. Drop it---replace $\floor{x}$ by $x$---and the system collapses
to a \emph{linear} recurrence: with $\Sigma_k=X_k+Y_k$, $\Delta_k=Y_k-X_k$ one gets
$\Sigma_k=\tfrac{k-1}{k}\Sigma_{k-1}$, $\Delta_k=\tfrac{k+1}{k}\Delta_{k-1}$, hence
$k(X_k+Y_k)=2M-1$ and $Y_k-X_k=\tfrac{k+1}2$, so $M$ falls out at the endpoint. But
the answer is \emph{wrong}: this floorless model is forced onto the diagonal
$i/N\to\tfrac12$, where it predicts $M/N^2\to\tfrac34$, against the true value
$\mathcal C(\tfrac12)=(2+\pi)^2/(16\pi)=0.526\ldots$. The discarded remainders
$r^X_k,r^Y_k\in[0,k)$ of \eqref{eq:mass}, accumulated over $\Theta(N)$ steps, carry
the whole dependence on $\alpha=i/N$ and shift the constant; the floor cannot simply
be discarded.

\subsection*{Change of variable: sum the area by rows}

The cure is not to smooth the floor away but to change variables so it can be
accounted for \emph{exactly}. Sum the same area \eqref{eq:Msum} \emph{by rows
instead of by columns}. The drops are nonincreasing staircases (Lemma~\ref{lem:ab}),
descending from $a_2\approx M/2$ to $0$, and for such a staircase the tail sum
equals a sum of \emph{level-widths}: each unit of height belongs to one level, so
totalling the heights is the same as totalling how far each level reaches
(Figure~\ref{fig:conjugate}). For $m\ge1$ define the \emph{widths} as counts,
\begin{equation}\label{eq:width}
  R_m:=1+\sum_{k\ge2}[a_k\ge m],\qquad S_m:=1+\sum_{k\ge2}[b_k\ge m]
\end{equation}
(the Iverson bracket $[P]$ is $1$ if $P$ holds, $0$ otherwise; since $a_k=b_k=0$ for
$k>N$, both sums are finite). By monotonicity, when the level is reached ($a_2\ge m$)
this is $R_m=\max\{k:a_k\ge m\}$, the last order at which $a_k$ reaches level $m$---the
order where the staircase crosses below height $m$ ($a_{R_m}\ge m>a_{R_m+1}$); when it
is not ($a_2<m$), the sum is empty and $R_m=1$. Because $a_k$ is monotone, this single
index is also the \emph{count} of orders reaching the level: $R_m$ is at once a
\emph{width} (how many $k$ have $a_k\ge m$) and an \emph{index} (the last such $k$).
The width reading uses the first face; the recurrence below is read off the drop
recurrence \eqref{eq:abrec} \emph{at that index}, and this two-facedness is exactly
what closes it. Writing
$a_k=\sum_{m\ge1}[a_k\ge m]$ and swapping the sums gives the layer-cake
$\sum_k a_k=\sum_m\sum_k[a_k\ge m]=\sum_m(R_m-1)$, so \eqref{eq:Msum} becomes the
\emph{width reading}
\begin{equation}\label{eq:MsumRS}
  M=N+\sum_{m\ge1}(S_m-1)=i+1+\sum_{m\ge1}(R_m-1)
\end{equation}
---the same staircase summed along its rows (Figure~\ref{fig:conjugate}).

And now the floor is harmless. At $k=R_m$ the non-local tail sum in
\eqref{eq:abrec} \emph{truncates}: because the widths decrease, only the low levels
reach past the cut (Figure~\ref{fig:truncate}), so what was the floor of a whole tail
collapses to a finite, lower-triangular sum. This turns \eqref{eq:abrec} into
the \emph{exact} width recurrence
\begin{equation}\label{eq:renewintro}
  (2m-1)R_m=N+\sum_{r=1}^{m-1}S_r+E_m,\qquad
  2mS_m=i+\sum_{r=1}^{m}R_r+F_m,
\end{equation}
with bounded corrections $1-m\le E_m\le m-1$, $1-m\le F_m\le m$
(Lemma~\ref{lem:filter}). Two displays now carry the section: the formula
\eqref{eq:MsumRS} giving $M$ from the widths, and the recurrence
\eqref{eq:renewintro}---read off \eqref{eq:abrec} at $k=R_m$---giving the
widths themselves. Crucially $E_m,F_m$ are the floor's remainders, now collected
\emph{one per level} (size $O(m)$) rather than one per step: being bounded, they
vanish under the rescaling of \S\ref{sec:wallis}, so this time the smoothing is
legitimate.

\begin{lemma}[Interlacing]\label{lem:interlace}
The conjugate widths satisfy
\begin{equation}\label{eq:interlace}
  R_{m+1}\le S_m\le R_m,\qquad R_1=N.
\end{equation}
\end{lemma}

\begin{proof}
If $a_k\ge m+1$ then $b_k\ge a_k-1\ge m$ by \eqref{eq:abdiff}; hence
$\{k:a_k\ge m+1\}\subseteq\{k:b_k\ge m\}\subseteq\{k:a_k\ge m\}$, and taking maxima
gives $R_{m+1}\le S_m\le R_m$. Finally $Y_N=N$ gives
$a_{N+1}=\floor{Y_N/(N+1)}=0$, while $Y_{N-1}=Y_N+b_N\ge N$ gives $a_N\ge1$; hence
$R_1=N$.
\end{proof}

\begin{figure}[t]
\centering
\includegraphics[width=\textwidth]{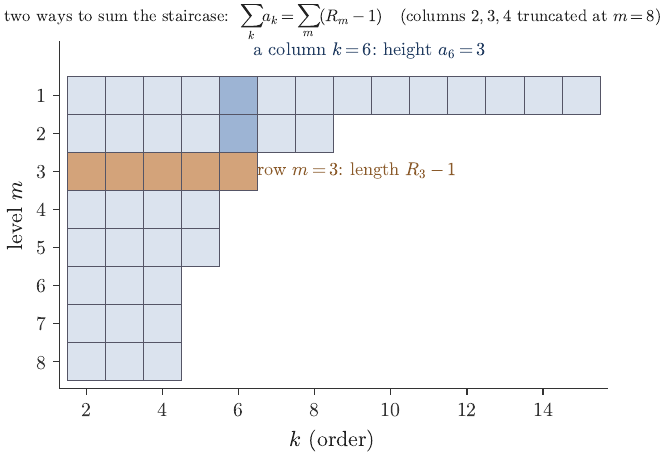}
\caption{Summing a staircase two ways, for $M=100$. The column of height $a_k$ over
order $k$ has its $m$-th row reaching to $R_m$; counting cells down the columns gives
$\sum_k a_k$, along the rows $\sum_m(R_m-1)$, and the two agree---the transpose behind
the two readings $M=i+1+\sum_k a_k=i+1+\sum_m(R_m-1)$ of \eqref{eq:Msum}. (Columns
$k=2,3,4$, heights $50,17,8$, truncated.)}
\label{fig:conjugate}
\end{figure}

\begin{example}\label{ex:filters}
For $M=100$ (Figure~\ref{fig:conjugate}) the widths are
\[
  \begin{array}{c|cccccc}
    m   & 1 & 2 & 3 & 4 & 5 & 6\\\hline
    R_m & 15 & 8 & 6 & 5 & 5 & 4\\
    S_m & 10 & 7 & 6 & 5 & 4 & 4
  \end{array}
\]
displaying the interlacing $R_{m+1}\le S_m\le R_m$ of Lemma~\ref{lem:interlace}
(e.g.\ $R_2=8\le S_1=10\le R_1=15$).
\end{example}

\begin{lemma}\label{lem:diverge}
For each fixed $m$, $R_m\to\infty$ and $S_m\to\infty$ as $M\to\infty$.
\end{lemma}

\begin{proof}
Put $k=R_m$. Since $a_{k+1}<m$ we have $Y_k<m(k+1)$, and $X_k<Y_k$, so by the
lower bound in \eqref{eq:mass}, $2M-1\le k(X_k+Y_k)<2mk(k+1)$. As $M\to\infty$
this forces $R_m\to\infty$; and $S_m\ge R_{m+1}$ by \eqref{eq:interlace}.
\end{proof}

So far the index $k=R_m$ gives only the \emph{inequality} $a_{R_m}\ge m>a_{R_m+1}$:
the staircase has passed below level $m$, but it may have \emph{skipped}---overshooting
the level ($a_{R_m}>m$) or falling through several levels in one step
($a_{R_m+1}\ll m-1$). Near the top, where the drops are large ($a_2\approx M/2$) and
coarse, skipping is the rule; far out, where the drops are small and change slowly, the
descent is forced to be \emph{clean}---landing exactly on level $m$ and stepping down by
one. The threshold is the width being long enough, $R_m\ge2m+1$.

\begin{lemma}[No skipped levels]\label{lem:noskip}
For every $m\ge1$, provided the width is long enough:
\begin{equation}\label{eq:noskip}
  R_m\ge 2m+1\ \Rightarrow\ a_{R_m}=m,\ a_{R_m+1}=m-1;\qquad
  S_m\ge 2m+2\ \Rightarrow\ b_{S_m}=m,\ b_{S_m+1}=m-1.
\end{equation}
In particular, for each \emph{fixed} $m$ both hypotheses hold once $N$ is large
(Lemma~\ref{lem:diverge}), so \eqref{eq:noskip} then holds unconditionally.
\end{lemma}

\begin{proof}
Let $k=R_m$ and $A=a_k\ge m$, while $a_{k+1}<m$. Since $b_k\le a_k=A$,
\begin{equation}\label{eq:Ylow}
  Y_k=Y_{k-1}-b_k\ge Ak-A=A(k-1),
\end{equation}
whereas $a_{k+1}<m$ gives
\begin{equation}\label{eq:Yhigh}
  Y_k<m(k+1).
\end{equation}
If $A\ge m+1$ then \eqref{eq:Ylow}--\eqref{eq:Yhigh} give
$(m+1)(k-1)<m(k+1)$, i.e.\ $k<2m+1$, contradicting the hypothesis $R_m=k\ge2m+1$; so
$A=m$. Now \eqref{eq:Ylow} reads $Y_k\ge m(k-1)$, and for $k\ge 2m-1$ we have
$m(k-1)\ge(m-1)(k+1)$; together with \eqref{eq:Yhigh}, $(m-1)(k+1)\le Y_k<m(k+1)$, so
$a_{k+1}=m-1$. This proves the first pair.

For the second, let $k=S_m$ and $B=b_k$. From \eqref{eq:abdiff}, $a_k\le B+1$, so
\begin{equation}\label{eq:Xlow}
  X_k=X_{k-1}-a_k\ge Bk-(B+1)=B(k-1)-1,
\end{equation}
while $b_{k+1}<m$ gives $X_k<m(k+1)$. If $B\ge m+1$ these force $k<2m+2$,
contradicting $S_m=k\ge2m+2$; hence $B=m$. For $k\ge 2m$ one has
$m(k-1)-1\ge(m-1)(k+1)$, and with $X_k<m(k+1)$ this yields $b_{k+1}=m-1$.
\end{proof}

Which levels are clean? Since $R_m\asymp N/\sqrt m$, the hypothesis $R_m\ge2m+1$ reads
$m\lesssim N^{2/3}$: for a given cell the \emph{low} levels $1\le m\lesssim N^{2/3}$ all
land cleanly, while the higher ones---where the staircase is still steep---may skip.
Read the other way, a \emph{fixed} level $m$ becomes clean once $N\gtrsim m^{3/2}$, and
as $N\to\infty$ this clean band swells past any fixed $m$ (Lemma~\ref{lem:diverge}).

We use cleanness in exactly one place, at the top of this band. It lets us read $M$ off
a \emph{single} width: at $k=R_m$ the clean landing gives $Y_k\asymp mR_m$, so the mass
estimate \eqref{eq:mass} yields $2M\asymp R_m(X_k+Y_k)\asymp2mR_m^2$. This single-level
reading---turning one width back into the value $M$---is the sole use of no-skipping in
the paper (the width recurrence needs none of it); \S\ref{sec:sqrt} reads $M$ at one
level $m\asymp N^{2/3}$ and balances this against the one-cell rounding of $R_m$.

\begin{lemma}[Single-level reading]\label{lem:crossread}
At any level $m$ where the no-skipping lemma holds, with $k=R_m$,
\begin{equation}\label{eq:crossread}
  (m-2)R_m^2\ \le\ M\ <\ m R_m(R_m+1)+\tfrac12,
  \qquad\text{that is}\quad M=mR_m^2+O(R_m^2+mR_m).
\end{equation}
\end{lemma}

\begin{proof}
By Lemma~\ref{lem:noskip}, $a_{k+1}=m-1$, i.e.\ $(m-1)(k+1)\le Y_k<m(k+1)$; with
\eqref{eq:gap}, $2(m-1)(k+1)-k\le X_k+Y_k<2m(k+1)$, and \eqref{eq:mass} at this $k$
gives \eqref{eq:crossread}.
\end{proof}

That reads $M$ from one width; the widths themselves come from the same construction.
Because $a_\ell,b_\ell$ are nonincreasing, the tail sums \eqref{eq:tails} may be
summed by layers (with $(x)_+:=\max(x,0)$):
\begin{equation}\label{eq:layer}
  \sum_{\ell\ge k}b_\ell=\sum_{r\ge1}(S_r-k+1)_+,\qquad
  \sum_{\ell\ge k}a_\ell=\sum_{r\ge1}(R_r-k+1)_+ .
\end{equation}

The point is that at $k=R_m$ the non-local tail sums \eqref{eq:layer}
\emph{truncate} (Figure~\ref{fig:truncate}), exactly and with no hypothesis on $m$.
Fix $m$, put $k=R_m$, and let $A:=a_k=\floor{Y_{k-1}/k}\ge m$, so that $Y_{k-1}=Ak+\eta$
with $\eta\in\{0,\dots,k-1\}$. The number of $S$-levels that reach past the cut is the
drop $b_k$ itself:
\begin{equation}\label{eq:Tbk}
  \#\{r:S_r\ge k\}=b_k,\qquad\text{because } S_r\ge k\iff b_k\ge r .
\end{equation}
Write $B:=b_k$ and $e:=A-B\in\{0,1\}$ (by \eqref{eq:abdiff}); then $B\ge m-1$. Two
observations truncate the layer sum $\sum_{\ell\ge k}b_\ell=\sum_{r=1}^{B}(S_r-k+1)$:
every level $r<m$ survives ($S_r\ge R_{r+1}\ge R_m=k$ by \eqref{eq:interlace}), and
every $r\in[m,B]$ has $S_r=k$ exactly (there $R_r=k$, since $a_k=A\ge r$ while
$a_{k+1}<m\le r$, so $k\le S_r\le R_r=k$). Hence the tail collapses to
$\sum_{\ell\ge k}b_\ell=\sum_{r<m}(S_r-k+1)+(B-m+1)$, and equating with $Y_{k-1}=Ak+\eta$,
then collecting the multiples of $k=R_m$, gives the first width equation
\emph{exactly}:
\begin{equation}\label{eq:Eident}
  (2m-1)R_m=N+\sum_{r=1}^{m-1}S_r+E_m,\qquad
  E_m=m-D_m,\quad D_m:=(A-m)(k-1)+\eta+e .
\end{equation}
It is closed and lower-triangular, and the correction is \emph{sharply} bounded:
\begin{equation}\label{eq:Ebound}
  1\le D_m\le 2m-1,\qquad\text{so}\qquad 1-m\le E_m\le m-1 .
\end{equation}
For $D_m\ge1$: if $e=1$ this is immediate; if $e=0$ then $a_k=b_k$, and the gap
\eqref{eq:gap} forces $\eta\ge1$ (writing $X_{k-1}=Bk+\xi$, one has
$Y_{k-1}-X_{k-1}=\eta-\xi\ge1$). For $D_m\le2m-1$: since $a_{k+1}<m$, $Y_k<m(k+1)$;
but $Y_k=Y_{k-1}-B=A(k-1)+\eta+e$, so $D_m=Y_k-m(k-1)<2m$, an integer. The skip $A-m$
may be large, yet \eqref{eq:Ebound} makes it cost nothing: a larger skip inflates
$Y_{k-1}$ and the survivor count $B$ in lock-step. \big(For $M=100$, $m=3$: $k=R_3=6$,
$A=a_6=3$, $B=b_6=3$, $e=0$, $\eta=Y_5-18=5$; so $D_3=5$ and $E_3=-2$, matching
$5\cdot6=30=N+S_1+S_2+E_3=15+17-2$.\big)

Because $\abs{E_m}\le m-1<\tfrac12(2m-1)$, the correction is exactly a rounding
remainder: with $\langle x\rangle:=\floor{x+\tfrac12}$ the nearest integer,
\eqref{eq:Eident} says $R_m=\big\langle(N+\sum_{r<m}S_r)/(2m-1)\big\rangle$. The
$b$-staircase gives the mirror statement, and we record the pair together.

\begin{figure}[t]
\centering
\includegraphics[width=0.80\textwidth]{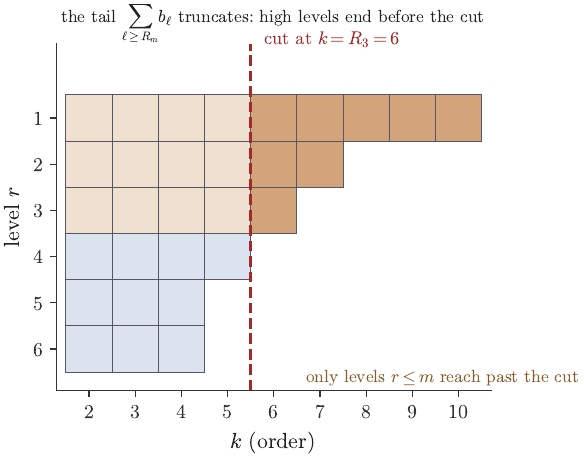}
\caption{Why the coupling closes (the $b$-staircase for $M=100$, cut at $k=R_3=6$).
The tail $\sum_{\ell\ge R_m}b_\ell$ is shaded. Since the widths $S_r$ decrease, only
low levels reach past the cut: the levels $r<m$ survive with full width, while a few
borderline $r\ge m$ contribute a single cell $S_r=R_m$, absorbed into $E_m$. Either
way the non-local tail collapses to a finite lower-triangular sum.}
\label{fig:truncate}
\end{figure}

\begin{lemma}[The width recurrence as nearest-integer rounding]\label{lem:filter}
For \emph{every} $m\ge1$ and every cell $(i,j)$, with $\langle x\rangle:=\floor{x+\tfrac12}$,
\begin{equation}\label{eq:renewround}
  R_m=\left\langle\frac{N+\sum_{r=1}^{m-1}S_r}{\,2m-1\,}\right\rangle,\qquad
  S_m=\left\langle\frac{i+\sum_{r=1}^{m}R_r}{\,2m\,}\right\rangle .
\end{equation}
Equivalently, in terms of the rounding remainders $E_m,F_m$,
\begin{equation}\label{eq:renewR}
  (2m-1)R_m=N+\sum_{r=1}^{m-1}S_r+E_m,\qquad
  2mS_m=i+\sum_{r=1}^{m}R_r+F_m,
\end{equation}
with $1-m\le E_m\le m-1$ and $1-m\le F_m\le m$.
\end{lemma}

The width process is therefore a deterministic, lower-triangular
\emph{nearest-integer system}: $R_m$ is the rounded average of the datum $N$ and the
lower widths, $S_m$ likewise. This is exact---no restriction on $m$ relative to $N$,
no no-skipping---and it is the structure that makes \S\ref{sec:wallis} clean. Removing
only the rounding leaves a system solvable in closed form; restoring it
(\S\ref{sec:sqrt}) moves each width by less than one.

\begin{proof}
The first equation of \eqref{eq:renewround} is \eqref{eq:Eident}--\eqref{eq:Ebound}:
since $\abs{E_m}\le m-1$, the integer $R_m$ lies within $\tfrac{m-1}{2m-1}<\tfrac12$ of
$(N+\sum_{r<m}S_r)/(2m-1)$, hence is its nearest integer. The second is the mirror
computation on the $b$-staircase. Put $k=S_m$, $B:=b_k=\floor{X_{k-1}/k}\ge m$,
$X_{k-1}=Bk+\xi$, and $A:=a_k=B+e$ with $e\in\{0,1\}$. The $a$-tail survivors number
$\#\{r:R_r\ge k\}=a_k=A$ (as $R_r\ge k\iff a_k\ge r$); all $r\le m$ survive
($R_r\ge R_m\ge S_m=k$) and each $r\in(m,A]$ has $R_r=k$ ($R_r\le R_{m+1}\le S_m=k\le
R_r$). Hence $\sum_{\ell\ge k}a_\ell=\sum_{r\le m}(R_r-k+1)+(A-m)$, and equating with
$X_{k-1}=Bk+\xi$,
\[
  2mS_m=i+\sum_{r=1}^{m}R_r+F_m,\qquad F_m=m-D'_m,\quad D'_m:=(B-m)(k-1)+\xi-e .
\]
Here $0\le D'_m\le2m-1$, so $1-m\le F_m\le m$. For $D'_m\ge0$: if $e=0$ it is clear;
if $e=1$ the gap \eqref{eq:gap} forces $\xi\ge1$ (then $Y_{k-1}-X_{k-1}=k+\eta'-\xi\le
k-1$ needs $\xi\ge\eta'+1$). For $D'_m\le2m-1$: $b_{k+1}<m$ gives $X_k<m(k+1)$ with
$X_k=B(k-1)+\xi-e$, so $D'_m=X_k-m(k-1)<2m$. As $\abs{F_m}\le m=\tfrac12\cdot2m$ and
$\langle\cdot\rangle$ rounds half up, $S_m=\big\langle(i+\sum_{r\le m}R_r)/(2m)\big\rangle$.
(At the empty levels $R_m=1$ or $S_m=1$ the identities hold with $E_m=m-M$,
$F_m=m-M+1$, where $M=\Tar_{i,j}$ is the value itself: there $S_r=1$ for $r\ge m$, so
the width reading \eqref{eq:MsumRS} gives $N+\sum_{r<m}S_r-(m-1)=M$. The two ranges
differ: $R_m=1$ forces $\floor{M/2}<m$, i.e.\ $M\le2m-1$, so $E_m=m-M\in[1-m,m-1]$;
$S_m=1$ forces $\floor{(M-1)/2}<m$, i.e.\ $M\le2m$, so $F_m=m-M+1\in[1-m,m]$.)
\end{proof}

\section{The two Wallis modes}\label{sec:wallis}

The width recurrence (Lemma~\ref{lem:filter}) is an exact nearest-integer system.
Drop the rounding, and what remains is a plain linear system with an \emph{explicit}
closed-form solution---and it is that solution, tracked back to the true $R_m,S_m$ to
within one cell per level (\S\ref{sec:sqrt}), that carries the whole theorem. Let
$(\bar R_m,\bar S_m)$ solve the error-free system
\begin{equation}\label{eq:barsys}
  (2m-1)\bar R_m=N+\sum_{r=1}^{m-1}\bar S_r,\qquad
  2m\,\bar S_m=i+\sum_{r=1}^{m}\bar R_r\qquad(m\ge1).
\end{equation}
Being lower-triangular, it has a unique solution; we write it down.

Introduce the central binomial fractions and their reciprocal partners
\begin{equation}\label{eq:cw}
  c_m:=\frac{1}{4^m}\binom{2m}{m}\quad(c_0=1),\qquad
  w_m:=\frac{1}{2m\,c_m}=\prod_{r=1}^{m-1}\frac{2r}{2r+1}\quad(w_1=1).
\end{equation}
The one-line ratio $c_m/c_{m-1}=(2m-1)/2m$ gives $(2m-1)c_{m-1}=2m\,c_m$, so the
reciprocal partner is $w_m=1/(2m\,c_m)=1/\big((2m-1)c_{m-1}\big)$.

\begin{lemma}[Two modes]\label{lem:modes}
With $N=(j+1)+i$, the error-free system \eqref{eq:barsys} has the unique solution
\begin{equation}\label{eq:modes}
  \bar R_m=(j+1)\,c_{m-1}+i\,w_m,\qquad
  \bar S_m=(j+1)\,c_m+i\,w_m .
\end{equation}
\end{lemma}

\begin{proof}
Run \eqref{eq:barsys} as a forward sweep: with $A_m:=N+\sum_{r=1}^{m}\bar S_r$ and
$B_m:=i+\sum_{r=1}^{m}\bar R_r$ (so $A_0=N$, $B_0=i$), the two equations read
\begin{equation}\label{eq:ABsweep}
  \bar R_m=\frac{A_{m-1}}{2m-1},\quad B_m=B_{m-1}+\bar R_m;\qquad
  \bar S_m=\frac{B_m}{2m},\quad A_m=A_{m-1}+\bar S_m .
\end{equation}
Two linear combinations \emph{diagonalise} the sweep. Set $G_m:=A_m-B_m$ and
$K_m:=B_m-2mG_m$. From \eqref{eq:ABsweep}, $G_m=G_{m-1}+(\bar S_m-\bar R_m)$, and
eliminating $A_{m-1},B_m$ collapses the increment to $\bar S_m-\bar R_m=-G_{m-1}/(2m)$;
the same elimination (using $A_{m-1}=K_{m-1}+(2m-1)G_{m-1}$) gives the $K$-increment.
Thus the two combinations obey \emph{decoupled scalar} recurrences,
\begin{equation}\label{eq:GK}
  G_m=\frac{2m-1}{2m}\,G_{m-1},\qquad K_m=\frac{2m}{2m-1}\,K_{m-1},
\end{equation}
with initial data $G_0=N-i=j+1$ and $K_0=i$. Hence, with $c_m$ as in \eqref{eq:cw},
\[
  G_m=(j+1)\,c_m,\qquad K_m=\frac{i}{c_m}.
\]
Finally $A_{m-1}=K_{m-1}+(2m-1)G_{m-1}$ recovers the widths:
\[
  \bar R_m=\frac{A_{m-1}}{2m-1}=G_{m-1}+\frac{K_{m-1}}{2m-1}=(j+1)\,c_{m-1}+i\,w_m,
\]
using $w_m=1/\big((2m-1)c_{m-1}\big)=1/(2m\,c_m)$; and $B_m=K_m+2mG_m$ gives
$\bar S_m=G_m+K_m/(2m)=(j+1)\,c_m+i\,w_m$.
\end{proof}

So the two boundary data have \emph{diagonalised}: the gap datum $j+1$ excites the
contracting Wallis mode $G_m=(j+1)c_m$, and the row datum $i$ its reciprocal
$K_m=i/c_m$. Restored to the widths, the $K$-mode contributes the same $i\,w_m$ to both
$\bar R_m,\bar S_m$, while the $G$-mode contributes the neighbouring $c_{m-1},c_m$.
That is the sense in which the profile carries \emph{two Wallis modes}.

The true widths never stray far from this profile---in fact each stays within a
\emph{single cell} of it, at every level.

\begin{lemma}[One-cell tracking]\label{lem:track}
For every $m\ge1$ and every cell,
\begin{equation}\label{eq:onecell}
  \abs{R_m-\bar R_m}<1,\qquad \abs{S_m-\bar S_m}<1 .
\end{equation}
Thus each true width is one of the two integers bracketing its Wallis prediction; if
$\bar R_m$ or $\bar S_m$ is itself an integer, the width equals it.
\end{lemma}

\begin{proof}
Set $p_m:=R_m-\bar R_m$ and $q_m:=S_m-\bar S_m$. The profile solves the error-free
system \eqref{eq:barsys} exactly, so subtracting $\bar R_m=(N+\sum_{r<m}\bar S_r)/(2m-1)$
from the nearest-integer form \eqref{eq:renewround} of the true width,
\begin{equation}\label{eq:strayround}
  R_m=\Big\langle\bar R_m+\frac{\sum_{r<m}q_r}{2m-1}\Big\rangle,\qquad
  S_m=\Big\langle\bar S_m+\frac{\sum_{r\le m}p_r}{2m}\Big\rangle .
\end{equation}
We show $\abs{p_m},\abs{q_m}<1$ by induction. \emph{Base:} $p_1=0$, since $R_1=N$
(Lemma~\ref{lem:interlace}) and $\bar R_1=(j+1)c_0+iw_1=N$; then $S_1=\langle\bar S_1+
p_1/2\rangle$ gives $\abs{q_1}\le\tfrac12$. \emph{Step:} assume $\abs{p_r},\abs{q_r}<1$
for $r<m$. Then the shift inside the first rounding is
$\big|\sum_{r<m}q_r/(2m-1)\big|<(m-1)/(2m-1)<\tfrac12$; since $\langle\cdot\rangle$
moves its argument by at most $\tfrac12$, \eqref{eq:strayround} gives
\[
  \abs{p_m}\le\Big|\frac{\sum_{r<m}q_r}{2m-1}\Big|+\frac12<\frac12+\frac12=1 .
\]
With this, $\big|\sum_{r\le m}p_r/(2m)\big|<m/(2m)=\tfrac12$, and the same step on the
$S$-side gives $\abs{q_m}<1$. The averaging denominators $2m-1$ and $2m$ exactly match
the number of accumulated strays, so the errors are washed out, level by level.
\end{proof}

The classical Wallis asymptotics of the central binomial coefficient,
\begin{equation}\label{eq:wallisasy}
  c_m=\frac{1}{\sqrt{\pi m}}\Big(1+O(m^{-1})\Big),\qquad
  w_m=\frac{1}{2m\,c_m}=\frac{\sqrt\pi}{2\sqrt m}\Big(1+O(m^{-1})\Big),
\end{equation}
are the sole entry point of $\pi$: they are a property of the recursion's own
coefficients $(2r-1)/2r$, owing nothing to the two edge theorems. Feeding
\eqref{eq:wallisasy} into \eqref{eq:modes},
\begin{equation}\label{eq:sqrtprofile}
  \sqrt m\,\bar R_m
  =(j+1)\sqrt m\,c_{m-1}+i\sqrt m\,w_m
  =\frac{j+1}{\sqrt\pi}+\frac{i\sqrt\pi}{2}+O\!\Big(\frac Nm\Big).
\end{equation}

This line already contains the theorem. Reading $M$ at a single level $m$
gives $\sqrt M\approx\sqrt m\,R_m\approx\sqrt m\,\bar R_m$ (made precise in
\S\ref{sec:sqrt}), so
\[
  \sqrt{M}\ \approx\ \frac{j+1}{\sqrt\pi}+\frac{i\sqrt\pi}{2}
  \ =\ \frac{\pi i+2j+2}{2\sqrt\pi},
\]
whose square is $(\pi i+2j+2)^2/(4\pi)$. Thus it is the \emph{square root} of the
density coefficient that is linear. Define, explicitly,
\begin{equation}\label{eq:Cdef}
  \mathcal C(\alpha):=\frac{[\,2+(\pi-2)\alpha\,]^2}{4\pi}\qquad(0\le\alpha\le1),
\end{equation}
so that Theorem~\ref{thm:main} reads $\Tar_{i,j}=N^2\mathcal C(i/N)+O(N^{4/3})$; then
\begin{equation}\label{eq:sqrtlaw}
  \sqrt{\mathcal C(\alpha)}=(1-\alpha)\,\underbrace{\tfrac{1}{\sqrt\pi}}_{\sqrt{\mathcal C(0)}}
  +\alpha\,\underbrace{\tfrac{\sqrt\pi}{2}}_{\sqrt{\mathcal C(1)}},
\end{equation}
the straight-line interpolation between the two edge rates $\sqrt{\mathcal C(0)}=1/\sqrt\pi$
(row~$0$) and $\sqrt{\mathcal C(1)}=\sqrt\pi/2$ (column~$0$). This linear-in-$\sqrt{\ }$ law
is why the level contours are straight (\S\ref{sec:contour}); the centred form
$\pi i+2j+2=2N+(\pi-2)i$ it produces is natural. Replacing the $+2$ by any constant
$+c$ shifts the main term by $O(N)$, absorbed into the error; we keep $+2$ because it
is exactly the value the Wallis modes \eqref{eq:modes} deliver, not because it is the
optimal centring.

\section{Proof of the square-root law}\label{sec:sqrt}

We now restore the errors and prove Theorem~\ref{thm:main} in the equivalent
\emph{square-root form}
\begin{equation}\label{eq:sqrtmain}
  \sqrt{\Tar_{i,j}}=L_{i,j}+O\!\big(N^{1/3}\big),\qquad
  L_{i,j}:=\frac{\pi i+2j+2}{2\sqrt\pi}=\frac{i\sqrt\pi}{2}+\frac{j+1}{\sqrt\pi} .
\end{equation}
This is equivalent to \eqref{eq:main}: since $L_{i,j}\asymp N$ uniformly,
$\Tar_{i,j}-L_{i,j}^2=(\sqrt{\Tar_{i,j}}-L_{i,j})(\sqrt{\Tar_{i,j}}+L_{i,j})
=O(N^{1/3})\cdot O(N)=O(N^{4/3})$, and $L_{i,j}^2=(\pi i+2j+2)^2/(4\pi)$; the
converse is the same computation.

Why not simply sum the exact area \eqref{eq:MsumRS}? Because
$M=i+1+\sum_m(R_m-1)$ runs over $\asymp a_2=\floor{M/2}\asymp N^2$ nonempty levels, and
replacing each $R_m$ by its profile $\bar R_m$ commits a one-cell error \emph{per
level}; these have nonzero mean and accumulate to $\Theta(N^2)$---worse than the
trivial bound. The remedy is to read $M$ off a \emph{single} level, sharply, so nothing
accumulates.

To recover $M$ we do exactly that: the single-level reading
$M\asymp mR_m^2$ (Lemma~\ref{lem:crossread}, established in \S\ref{sec:filter})
turns one width $R_m$ into the value.
Which level $m$ should we use? Two errors pull in opposite directions
(Figure~\ref{fig:error}).
\begin{enumerate}[label=\textup{(\arabic*)},leftmargin=2.4em]
\item \emph{The reading is only a sandwich.} Lemma~\ref{lem:crossread} traps $M$
between $(m-2)R_m^2$ and about $mR_m^2$, a gap of about $2R_m^2\asymp N^2/m$ (since
$R_m\asymp N/\sqrt m$). Its \emph{relative} size is $\asymp1/m$, so a \emph{higher}
level reads $M$ more sharply---the gap shrinks like $N^2/m$.
\item \emph{The width is quantised.} We never know $R_m$ exactly, only that it lies
within one cell of the smooth profile, $\abs{R_m-\bar R_m}<1$
(Lemma~\ref{lem:track}). Through $M\asymp mR_m^2$ this one-cell rounding is amplified:
a change $dR_m$ moves $M$ by $\approx 2mR_m\,dR_m$, so the sub-unit error costs
$\asymp mR_m\asymp m\cdot N/\sqrt m=N\sqrt m$ in $M$---and this \emph{grows} with the
level.
\end{enumerate}
Error (1), $\asymp N^2/m$, wants $m$ large; error (2), $\asymp N\sqrt m$, wants $m$
small. They balance when $N^2/m=N\sqrt m$, i.e.\ $m^{3/2}=N$, i.e.\ at
$m\asymp N^{2/3}$; there both equal $N^{4/3}$, so $M$ is pinned to $O(N^{4/3})$, and
dividing by $2\sqrt M\asymp N$ gives $O(N^{1/3})$ for $\sqrt M$. That balance is the
exponent. (The reading is valid at this level: no-skipping holds there, since
$R_m\asymp N/\sqrt m\gtrsim m$ at $m\asymp N^{2/3}$---checked in the proof.)

\begin{figure}[t]
\centering
\includegraphics[width=0.72\textwidth]{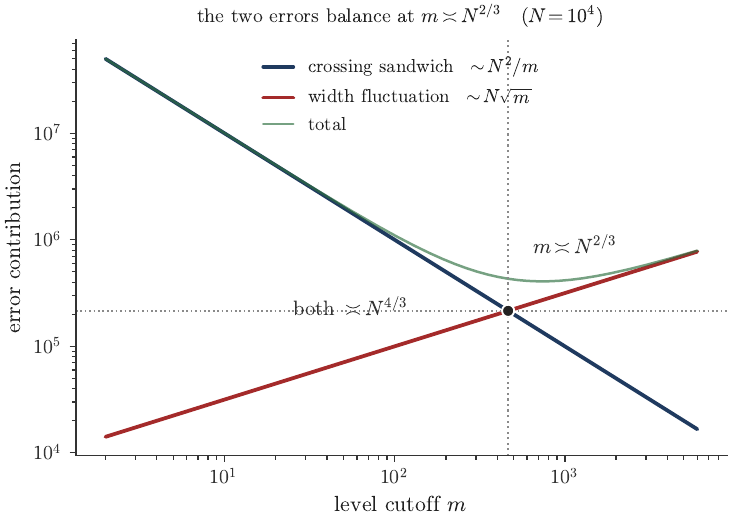}
\caption{Where the exponent comes from. Reading $M$ at level $m$ leaves a reading
sandwich of half-width $\asymp N^2/m$ (decreasing), while the one-cell rounding of the
width contributes $\asymp N\sqrt m$ (increasing). The two errors cross at $m\asymp N^{2/3}$,
where both equal $\asymp N^{4/3}$ (that is $\asymp N^{1/3}$ inside the square root);
this optimum is where the proof reads $M$.}
\label{fig:error}
\end{figure}

Both ingredients are now in hand: the single-level reading $M\asymp mR_m^2$
(Lemma~\ref{lem:crossread}) that turns a width into the value, and the one-cell bound
$\abs{R_m-\bar R_m}<1$ (Lemma~\ref{lem:track}) that pins the width to its Wallis
profile. The theorem is their balance.

\begin{proof}[Proof of Theorem~\ref{thm:main}]
We fix the level explicitly and give an effective bound. Two inputs supply the
constants. First, the classical two-sided Wallis bounds
\begin{equation}\label{eq:wallistwo}
  \frac{1}{\sqrt{\pi(m+\tfrac12)}}<c_m<\frac{1}{\sqrt{\pi m}}\qquad(m\ge1),
\end{equation}
proved by induction from $c_m/c_{m-1}=(2m-1)/2m$, sharpen the asymptotics
\eqref{eq:wallisasy}. Fed through $w_m=1/(2mc_m)$ into the profile \eqref{eq:modes},
they give, with $L_{i,j}=(j+1)/\sqrt\pi+i\sqrt\pi/2$ as in \eqref{eq:sqrtmain}, the
explicit bracket
\begin{equation}\label{eq:profilebracket}
  \frac{N}{\sqrt{\pi m}}<\bar R_m,\qquad
  0<\sqrt m\,\bar R_m-L_{i,j}<\frac{N}{m-1}\qquad(m\ge2).
\end{equation}
(Both are termwise: $\sqrt m\,c_{m-1}\in\big(\tfrac1{\sqrt\pi},
\tfrac1{\sqrt\pi}\sqrt{\tfrac m{m-1}}\big)$ and $\sqrt m\,w_m=\tfrac1{2\sqrt m\,c_m}\in
\big(\tfrac{\sqrt\pi}2,\tfrac{\sqrt\pi}2\sqrt{\tfrac{m+1/2}m}\big)$ by
\eqref{eq:wallistwo}, and $\sqrt{1+x}\le1+x$ bounds the two strays.) Second, the
one-cell bound $\abs{R_m-\bar R_m}<1$ (Lemma~\ref{lem:track}). Note also
$L_{i,j}\le\tfrac{\sqrt\pi}2N<N$.

\emph{Fix the level:} $m:=\floor{\tfrac14 N^{2/3}}$, and take $N\ge N_0$; one may take
$N_0=300$ (this makes $m\ge2$ and validates the three ``$N\ge N_0$'' steps below---the
binding one is $m\ge10$, i.e.\ $N^{2/3}\ge40$; the finitely many smaller cells are
absorbed into the implied constant).

\emph{No-skipping holds here.} By \eqref{eq:profilebracket} and Lemma~\ref{lem:track},
$R_m>\bar R_m-1>N/\sqrt{\pi m}-1$. Since $m\le\tfrac14N^{2/3}$ gives
$N/\sqrt{\pi m}\ge\tfrac2{\sqrt\pi}N^{2/3}$, while $2m+2\le\tfrac12N^{2/3}+2$, we get
$R_m\ge2m+1$ for $N\ge N_0$; likewise $S_m\ge2m+2$. So Lemma~\ref{lem:noskip} applies,
and with it the single-level reading \eqref{eq:crossread}.

\emph{Read.} At this level $\sqrt m\,\bar R_m<L_{i,j}+N/(m-1)<N$, so $\bar R_m<N/\sqrt m$
and $R_m<1.2\,N/\sqrt m$ (Lemma~\ref{lem:track}, $N\ge N_0$). Lemma~\ref{lem:crossread}
gives $(m-2)R_m^2\le M<mR_m^2+mR_m+\tfrac12$; taking square roots
($\sqrt{1-2/m}\ge1-2/m$, $\sqrt{1+x}\le1+\tfrac x2$),
\begin{equation}\label{eq:readexp}
  \Big|\sqrt M-\sqrt m\,R_m\Big|<\frac{3N}{m}+\sqrt m .
\end{equation}
Adding $\abs{\sqrt m\,R_m-\sqrt m\,\bar R_m}=\sqrt m\,\abs{R_m-\bar R_m}<\sqrt m$ and the
second bound of \eqref{eq:profilebracket} (with $N/(m-1)\le2N/m$),
\begin{equation}\label{eq:allexp}
  \Big|\sqrt{\Tar_{i,j}}-L_{i,j}\Big|<\frac{5N}{m}+2\sqrt m .
\end{equation}

\emph{Balance.} With $m=\floor{\tfrac14 N^{2/3}}$ one has $N/m<8N^{1/3}$ and
$\sqrt m\le\tfrac12N^{1/3}$ for $N\ge N_0$, so \eqref{eq:allexp} becomes the explicit
bound
\[
  \Big|\sqrt{\Tar_{i,j}}-L_{i,j}\Big|<5\cdot8\,N^{1/3}+2\cdot\tfrac12N^{1/3}=41\,N^{1/3}.
\]
This is uniform in $i,j$ and is $O(N^{1/3})$, which is
\eqref{eq:sqrtmain}. Squaring gives \eqref{eq:main}. (The constant $41$ is far from
best possible; numerically the deviation is well below $N^{1/3}$, and the true error is
smaller still---see \S\ref{sec:remarks}.)
\end{proof}

Theorem~\ref{thm:main} proves Beluhov's conjecture, as
$\big((\pi i+2j+2)^2-(\pi i+2j)^2\big)/(4\pi)=O(N)$.

\begin{remark}[Consistency with the edges]\label{rem:edges}
Setting $j=0$ in \eqref{eq:main} gives
$\Tar_{i,0}=(\pi i+2)^2/(4\pi)+O(i^{4/3})=\tfrac{\pi}{4}i^2+O(i^{4/3})$; setting
$i=0$ gives $\Tar_{0,j}=(2j+2)^2/(4\pi)+O(j^{4/3})=(j+1)^2/\pi+O(j^{4/3})$. Both
reproduce the main terms of the edge estimates \eqref{eq:andersson} and
\eqref{eq:brolineloeb}, as a check. On the column edge our error $O(i^{4/3})$ matches
Andersson's; on the row edge our $O(j^{4/3})$ is \emph{weaker} than Broline--Loeb's
$O(j)$. Our contribution is thus the uniform interior estimate, not an improvement on
either edge. In the square-root law
\eqref{eq:sqrtlaw} the two edges are literally the two endpoints $\alpha=0,1$;
neither is used in the proof.

More is true in the geometry. Andersson's result is really about the \emph{counting
function} of the Flavius sieve, $2\sqrt{n/\pi}+O(n^{1/6})$, whose $O(n^{1/6})$ is the
column-edge form of the $O(i^{4/3})$ above. The boundary-layer width $O(V^{1/6})$ of
Corollary~\ref{cor:contour} carries the \emph{same} exponent $\tfrac16$ into the whole
interior: our estimate extends Andersson's edge strength, uniformly, to every
direction $\alpha\in[0,1]$, not just $\alpha=1$.
\end{remark}

\section{Contour linearity}\label{sec:contour}

This corollary is the theorem drawn as a picture. By the bijection~\ref{P:bijection} a
value equals the number of cells beneath its contour---in effect the area below
it---and \eqref{eq:main} says that area is $(\rho+2)^2/(4\pi)$, quadratic in
$\rho=\pi i+2j$. An area quadratic in $\rho$ forces the contour to be a straight
line $\rho=\text{const}$, exactly as the nested triangles of
Figure~\ref{fig:contour} show; the corollary merely measures how thick the
boundary layer around that line is.

\begin{proof}[Proof of Corollary~\ref{cor:contour}]
The argument uses only the pointwise estimate of Theorem~\ref{thm:main} to fix the sign
of $\Tar_{i,j}-V$, not the bijection~\ref{P:bijection} (which enters only the intuition
of Figure~\ref{fig:contour}).
Put $\rho=\pi i+2j$, so $\rho+2=2N+(\pi-2)i$ and $2N\le\rho+2\le\pi N$; then
\eqref{eq:main} reads $\Tar_{i,j}=(\rho+2)^2/(4\pi)+O((\rho+1)^{4/3})$. Set
$x=\rho+2$ and $s=2\sqrt{\pi V}$. Let $A$ be the implied constant of the
$O((\rho+1)^{4/3})=O(V^{2/3})$ error in \eqref{eq:main}, and set $K:=4\pi A/(3\sqrt\pi)$.
For cells with $\tfrac12 s\le x\le2s$, if $\abs{x-s}\ge KV^{1/6}$ then, since
$x+s\ge\tfrac32 s=3\sqrt{\pi V}$,
\[
  \frac{\abs{x^2-s^2}}{4\pi}=\frac{\abs{x-s}\,(x+s)}{4\pi}\ge\frac{3\sqrt\pi\,K}{4\pi}\,V^{2/3}
  =A\,V^{2/3},
\]
which meets the error term; so the sign of $\Tar_{i,j}-V=(x^2-s^2)/(4\pi)+O(V^{2/3})$ is
that of $x-s$, giving the stated inclusions with boundary $x=s\pm KV^{1/6}$. For the remaining cells the quadratic term already decides the sign:
if $x\le\tfrac12 s$ then $\Tar_{i,j}\le x^2/(4\pi)+O(x^{4/3})\le\tfrac14 V+o(V)<V$,
and if $x\ge2s$ then $\Tar_{i,j}\ge x^2/(4\pi)-O(x^{4/3})\ge 4V-o(V)>V$; both differ
from $V$ by $\Theta(V)$, far exceeding $O(V^{2/3})$.
\end{proof}

\section{Concluding remarks}\label{sec:remarks}

\subsection*{The finer error term} The proven $O(N^{4/3})$ is far from what the
numbers show. Exhaustively over the $1{,}125{,}750$ cells with $N\le1500$
($M\le1.8\times10^6$), the square-root deviation $\sqrt{\Tar_{i,j}}-L_{i,j}$ lies in
$[-1.02,\,1.24]$ with mean $\approx0.16$ (Table~\ref{tab:conj}); it is oscillatory in
sign, and the extremes sit near the two edges. Sampling further---to
$M\approx5\times10^7$, concentrating near the edges---widens the observed range only to
about $(-1.1,\,1.8)$, with no sign of divergence. In the interior the deviation is
nearly constant: binned by $\alpha=i/N$ over $N\in[800,1200]$, every decile mean falls
in $[0.156,\,0.180]$. This points to the sharper
\begin{equation}\label{eq:conjecture}
  \sqrt{\Tar_{i,j}}=L_{i,j}+O(1),\qquad\text{equivalently}\qquad
  \Tar_{i,j}=\frac{(\pi i+2j+2)^2}{4\pi}+O(N),
\end{equation}
under which the contour width of Corollary~\ref{cor:contour} would be $O(1)$ rather
than $O(V^{1/6})$.

There is a precedent for this last step. On the row edge $i=0$ the sequence is the
Tchoukaillon numbers, for which Erd\H os and Jabotinsky~\cite{ErdosJabotinsky} proved
$\Tar_{0,j}=(j+1)^2/\pi+O(j^{4/3})$ and \emph{conjectured} the $O(j)$ later established
by Broline--Loeb \eqref{eq:brolineloeb}. The exponent improvement $\tfrac43\to1$ we
conjecture in the interior is exactly the Erd\H os--Jabotinsky improvement on the edge,
now sought in every direction. Our theorem falls short of \eqref{eq:conjecture} because
its bound is a \emph{pointwise} one-cell estimate read at a single level; closing the
gap would require the correlations among the nearest-integer remainders $E_m,F_m$
across levels---equivalently, the cancellation in the floor residues $r^X_k,r^Y_k$ of
Lemma~\ref{lem:mass}---which the present argument does not control. We leave
\eqref{eq:conjecture} open.

\begin{table}[t]
\centering
\renewcommand{\arraystretch}{1.15}
\begin{tabular}{@{}lrrrrr@{}}
\hline
 & $i$ & $j$ & $\Tar_{i,j}$ & $L_{i,j}$ & $\sqrt{\Tar_{i,j}}-L_{i,j}$ \\ \hline
interior
   & $200$  & $200$  & $84533$    & $290.65$  & $+0.098$ \\
   & $1000$ & $1000$ & $2106026$  & $1450.98$ & $+0.235$ \\
   & $3000$ & $3000$ & $18939549$ & $4351.81$ & $+0.145$ \\
column $j{=}0$
   & $1000$ & $0$    & $787561$   & $886.79$  & $+0.655$ \\
row $i{=}0$
   & $0$    & $1000$ & $319762$   & $564.75$  & $+0.721$ \\
off-diagonal
   & $1500$ & $300$  & $2248767$  & $1499.16$ & $+0.427$ \\
   & $300$  & $1500$ & $1238902$  & $1112.72$ & $+0.343$ \\
near-edge extremes
   & $1417$ & $30$   & $1618639$  & $1273.27$ & $-1.016$ \\
   & $1329$ & $6$    & $1399453$  & $1181.74$ & $+1.240$ \\
\hline
\end{tabular}
\caption{Evidence for \eqref{eq:conjecture}: the square-root deviation
$\sqrt{\Tar_{i,j}}-L_{i,j}$ is $O(1)$ across sizes and directions. The interior stays
well under $1$; the extremes (last two rows) sit near the two edges.}
\label{tab:conj}
\end{table}

\subsection*{Local spacing} Theorem~\ref{thm:main} does not yield the local spacing
laws $\Tar_{i+1,j}-\Tar_{i,j}\sim(\pi i+2j)/2$ and
$\Tar_{i,j+1}-\Tar_{i,j}\sim(\pi i+2j)/\pi$: an $O(N^{4/3})$ global error permits
adjacent entries to differ from the model by as much, and even the sharper $O(N)$
edge estimates \eqref{eq:andersson}--\eqref{eq:brolineloeb} leave a difference of the
same order as the spacing. A proof of these pointwise laws would be of independent
interest.

\section*{Acknowledgments}
The author thanks Xicheng Yang for stimulating discussions that helped inspire
this work. The large language models ChatGPT (OpenAI) and Claude Code (Anthropic)
were used substantially throughout: in drafting and revising the exposition, in
designing and running the numerical experiments, and as interlocutors whose
discussion and suggestions helped shape and clarify the arguments. The author is
responsible for all statements and for their verification.


\end{document}